\documentclass[12pt]{amsart}

\usepackage{amssymb,upref,enumerate}
\usepackage{tikz}
\usepackage{amsmath,amssymb,amscd,amsthm,amsxtra}
\usepackage{latexsym}
\usepackage{yfonts,mathrsfs, esint}
\usepackage{amsmath}
\usepackage{bbm}           
\usepackage{bm} 
\usepackage{eso-pic,graphicx}

\calclayout

\usepackage{pdfsync}

\usepackage[colorlinks=true, bookmarks=true,pdfstartview=FitV, linkcolor=blue, citecolor=blue, urlcolor=blue]{hyperref}

\newcommand{\F}{\mathcal{F}}
\newcommand{\Rf}{\mathcal{R}}

\def\Xint#1{\mathchoice 
{\XXint\displaystyle\textstyle{#1}}%
{\XXint\textstyle\scriptstyle{#1}}%
{\XXint\scriptstyle\scriptscriptstyle{#1}}%
{\XXint\scriptscriptstyle\scriptscriptstyle{#1}}%
\!\int} 
\def\XXint#1#2#3{{\setbox0=\hbox{$#1{#2#3}{\int}$} 
\vcenter{\hbox{$#2#3$}}\kern-.5\wd0}} 
 
\def\dashint{\Xint-}
\def\avgint{\Xint-}

\newcommand{\ra}{\rightarrow}

\newcommand{\bey}{\begin{eqnarray*}}
\newcommand{\eey}{\end{eqnarray*}}
\newcommand{\ba}{\begin{align}}
\newcommand{\ea}{\end{align}}
\newcommand{\bea}{\begin{align*}}
\newcommand{\ena}{\end{align*}}
\newcommand{\be}{\begin{equation}}
\newcommand{\ee}{\end{equation}}
\newcommand{\R}{\mathbb R}
\newcommand{\Z}{\mathbb Z}

\newcommand{\Ca}{\mathcal C}
\newcommand{\D}{\mathcal D}

\newcommand{\M}{\mathcal M }

\newcommand{\Sp}{\mathcal S }
\newcommand{\Q}{\mathbb Q }

\newcommand{\bc}{\begin{center}}
\newcommand{\ec}{\end{center}}

\newcommand{\al}{\alpha}

\newcommand{\supp}{\mathrm{supp}}

\DeclareMathOperator*{\esssup}{ess\,sup}

\newtheorem{theorem}{Theorem}[section]
\newtheorem{lemma}[theorem]{Lemma}

\theoremstyle{definition}

\theoremstyle{remark}
\newtheorem{remark}[theorem]{Remark}

\numberwithin{equation}{section}
\allowdisplaybreaks

\begin{document}

\author{Theresa C. Anderson}
\address{Theresa C. Anderson \\ Department of Mathematics \\University of Wisconsin--Madison
  \\Madison, WI 53705, USA}
\email{tcanderson@math.wisc.edu}

\author{David Cruz-Uribe, OFS}
\address{David Cruz-Uribe, OFS \\ Department of Mathematics \\ University of Alabama \\
Tuscaloosa, AL 35487, USA}
\email{dcruzuribe@ua.edu}

\author{Kabe Moen}
\address{Kabe Moen \\Department of Mathematics \\ University of Alabama \\
  Tuscaloosa, AL 35487, USA}
\email{kabe.moen@ua.edu}


\title[Extrapolation]{Extrapolation in the scale of generalized
  reverse H\"older weights}
\date{June 14, 2017}

\begin{abstract}
  We develop a theory of extrapolation for weights that satisfy a
  generalized reverse H\"older inequality in the scale of Orlicz
  spaces.  This extends previous results by Auscher and
  Martell~\cite{AM} on limited range extrapolation.  We then provide several applications of our extrapolation techniques.  These applications include new and known two weight inequalities for
  linear and bilinear operators.
\end{abstract}

\thanks{The first author was supported by NSF grant
  DMS-1502464 and by an NSF Graduate Student Fellowship, and completed part of this work while in residence at MSRI in Spring 2017.  The second
  author is supported by NSF Grant 1362425 and research funds from the
  Dean of the College of Arts \& Sciences, the University of
  Alabama.  The third author is supported by Simons Collaborative
  Grant for Mathematicians 16-0427.  The authors would like to thank
  J.~M.~Martell for helpful discussions which revealed a subtle
  mistake in an early version of this paper.}

\maketitle

\section{Introduction}
\label{section:intro}

The theory of extrapolation is a powerful technique in harmonic
analysis:  in a nutshell, it shows that norm inequalities for an operator on the
weighted Lebesgue space $L^p(w)$, for any $1<p<\infty$ and $w$ in the
Muckenhoupt weight class $A_p$, are all a consequence of such an
inequality being true for a single value of $p$.  Extrapolation was
introduced by Rubio de Francia~\cite{rubiodefrancia82} more than
thirty years ago; since then it has been refined and developed in a
number of directions, including applications to Banach functions
spaces, bilinear inequalities and two weight norm inequalities.  We
refer the reader to~\cite{MR2797562} for a more detailed discussion of
extrapolation and its history.

In its classic form, extrapolation depends on the weight $w$ being in
the Muckenhoupt $A_p$ class, $1<p<\infty$:
\[ [w]_{A_p} = \sup_Q \left(\avgint_Q w\,dx\right)\left(\avgint_Q
  w^{1-p'}\,dx\right)^{p-1} < \infty, \]
where the supremum is taken over all cubes $Q$ with sides parallel to
the coordinate axes.  More recently, Auscher and Martell~\cite{AM}
introduced a ``limited range'' version of extrapolation that depends on
both the $A_p$ class of the weight and its reverse H\"older class
$RH_s$, $1<s<\infty$.  We say $w\in RH_s$ if
\[ [w]_{RH_s} = \sup_Q \left(\avgint_Q
  w\,dx\right)^{-1}\left(\avgint_Q w^s\,dx\right)^{\frac{1}{s}} <
\infty, \]
where the supremum is taken over all cubes $Q$.
(For a precise statement of their result, see Theorem~\ref{chemaextrap} below.)  

As an immediate
consequence of their result we prove an extrapolation
theorem that only depends on the reverse H\"older class of the
weight.  We state it here; for precise definitions of the
notation used, please see Section~\ref{section:extrapolation}.

\begin{theorem} \label{RHextrap} 
Let $0<q_0<\infty$ and $\F =\{(f,g)\}$ be a family of pairs of
measurable functions.  Suppose there exists $p_0$ with $0< p_0\leq
q_0$ such that for all $w\in RH_{(\frac{q_0}{p_0})'}$,
\begin{equation} \label{eqn:RHExtrapol1}
\|f\|_{L^{p_0}(w)}\leq C\|g\|_{L^{p_0}(w)} \qquad  (f,g)\in
\F.  
\end{equation}
Then for all $p$, $0<p<q_0$, and $w\in RH_{(\frac{q_0}{p})'}$ we have
\begin{equation} \label{eqn:RHExtrapol2}
\|f\|_{L^p(w)}\leq C\|g\|_{L^p(w)} \qquad \forall \ (f,g)\in \F.
\end{equation}
\end{theorem}

The goal of this paper is to prove a generalization of Theorem~\ref{RHextrap} to a
larger class of weights.  To state our results we use the theory of
Young functions and Orlicz norms; for precise definitions, see
Section~\ref{section:orlicz}.   Given a Young function $\Psi$, we say
$w\in RH_\Psi$ if 
\[ \sup_Q  \left(\avgint_Q
  w\,dx\right)^{-1}\|w\|_{\Psi,Q} < \infty, \]
where again the supremum is taken over all cubes $Q$.
This class was first introduced by Harboure {\em et
  al.}~\cite{HSV} with a different but essentially equivalent
definition.   When $\Psi(t)=t^s$, this reduces to the class $RH_s$ defined
above, and as we will show below, these weights also satisfy the
classical reverse H\"older inequality.   Intuitively, they differ in
that they capture different behavior at different values of the range
of the weight.  For instance, if $\Psi$ is the ``oscillatory'' Young
function
\[ \Psi(t) = t^{s+ a\sin(\log\log(e^e+t))}, \qquad 0<a<s-1, \]
then depending on its size, on some cubes  $w\in RH_\Psi$ behaves like
a weight in $RH_{s+a}$ and on others like a weight in $RH_{s-a}$.  
Using these weights we prove the following extrapolation theorem.  The
class $B_r$ is a growth condition on Young functions:  see
Section~\ref{section:orlicz} below for a definition.

\begin{theorem} \label{thm:wtdOrlicz-special}
Given $0<p_0 <q_0$, suppose that for a fixed $\Psi_0 \in
B_{(\frac{q_0}{p_0})'}$ and all $w \in RH_{\Psi_0}$, 
\begin{equation} \label{eqn:wtdOrlicz1} 
 \|f\|_{L^{p_0}(w)} \lesssim \|g\|_{L^{p_0}(w)}, 
\qquad (f,g) \in \F. 
\end{equation}
Define the Young function $\Psi$ by $\Psi_0(t)=\Psi(t^r)$ with $r=\frac{(q_0/p_0)'}{(q_0/p)'} <1$.  If $p_0<p<q_0$ and 
$w\in RH_\Psi$, then we have that
\begin{equation} \label{eqn:extrapol-concl}
\|f\|_{L^{p}(w)} \lesssim \|g\|_{L^{p}(w)}, 
\qquad (f,g) \in \F. 
\end{equation}
\end{theorem}
 
We actually prove a more general result when $p= q_0$ and also when $p_0 = q_0$; see
Theorem~\ref{thm:wtdOrlicz} below.  

\bigskip

To illustrate the utility of our extrapolation results, we give
applications to the study of two weight norm inequalities for linear and
bilinear Calder\'on-Zygmund singular integral operators, and to the
theory of one weight inequalities for the bilinear fractional integral
operator.   Two weight norm inequalities have been studied for many
years by a number of authors.  In recent years this problem has
received renewed attention because of its close connection to the
so-called ``$A_2$ conjecture'' for singular integral operators that
was proved by Hyt\"onen~\cite{hytonenP2010}.   In the
decade of work that led to the proof of this result, it became clear that in order to
get the desired sharp constant estimate, the problem had to be treated
as a two-weight problem, with the $A_p$ condition used only once at
some key step.  

The techniques used, particularly the dyadic sparse operators that were
introduced by Lerner~\cite{Lern}, have been applied to the study of
``$A_p$ bump'' conditions for two-weight norm inequalities.  This
approach to generalizing the two-weight $A_p$ condition was first
introduced by Neugebauer~\cite{MR687633} but was systematically
developed by P\'erez~\cite{perez94,perez95} (see
also~\cite{MR2797562} and \cite{Anderson}).   The following result was first conjectured
by P\'erez and the second author~\cite{cruz-uribe-perez02} and finally
proved by Lerner~\cite{Lern}.

\begin{theorem} \label{thm:double-bump}
Suppose $1<p<\infty$ and $\Phi$ and $\Psi$ are Young functions such
that $\bar{\Phi}\in B_{p'}$ and $\bar{\Psi}\in B_p$.  If $(u,v)$ is a
pair of weights that satisfies
\[ \sup_Q\|u\|_{\Phi,Q}\|v^{-1}\|_{\Psi,Q} < \infty, \]
then given any Calder\'on-Zygmund singular integral operator $T$,
\[ \|(Tf)u\|_{L^p} \lesssim \|fv\|_{L^p}. \]
\end{theorem}

Below (see Theorem~\ref{thm:CZO}) we give a new proof of this result using extrapolation.
Moreover, we prove a slight generalization, proving that $T$ satisfies
a two weight, Coifman-Fefferman type inequality:
\[ \|(Tf)u\|_{L^p} \lesssim \|M_{\bar{\Psi}}(fv)\|_{L^p}. \]
We also extend Theorem~\ref{thm:double-bump} to the
bilinear setting, proving the analogous result for bilinear
Calder\'on-Zygmund singular integral operators.  These are the natural
generalization of the linear operators, and have been considered by a
number of authors:  see~\cite{Grafakos:2002wu,Kenig:1999vk}.  One
weight norm inequalities were characterized
by~\cite{Grafakos:2004us,Lerner:2009dq}.  
Our results in the two weight case are new.  The exact
condition required depends on whether $p>1$ or $1/2<p \leq 1$:  see
Theorems~\ref{bilinearp>1} and~\ref{thm:bilinear-p<1}.  

\medskip

Finally, we consider weighted norm inequalities for the bilinear
fractional integral operator
 $$BI_\al(f,g)(x)=\int_{\R^n} \frac{f(x-y)g(x+y)}{|y|^{n-\al}}\,dy.$$
This operator is the fractional analog of the bilinear Hilbert
transform; for weighted norm inequalities and  a history of this
operator, see~\cite{CDiPO,HoangMoen,M2}.  The corresponding maximal operator is
$$BM_\al(f,g)(x)=\sup_{r>0}\frac{1}{(2r)^{n-\al}}\int_{[-r,r]^n}
|f(x-y)g(x+y)|\,dy.$$ 
We also recall the less singular
bilinear fractional operators:
$$\mathcal
I_\al(f,g)(x)=\int_{\R^{2n}}\frac{f(y)g(z)}{(|x-y|+|x-z|)^{2n-\al}}\,dydz,
\qquad 0<\al<2n,$$ 
and the associated maximal operator
$$\mathcal M_\al(f,g)(x)=\sup_{Q\ni x}|Q|^{\frac{\al}{n}}\dashint_Q
|f(y)|\,dy \cdot \dashint_Q |g(z)|dz\qquad 0\leq\al<2n.  $$
These operators are the fractional operators corresponding to bilinear
CZOs.

As an application of our extrapolation techniques,
we are able to prove a Coifman-Fefferman type inequality
relating $BI_\alpha$ to the the {\it less singular} bilinear maximal
operator, $\M_\al$.  Our result improves one that was first proved by the third author in~\cite{M2}.

The remainder of this paper is organized as follows.  In
Section~\ref{section:extrapolation} we give some preliminary information
about Muckenhoupt weights and extrapolation,  and then prove Theorem~\ref{RHextrap}.  In
Section~\ref{section:orlicz} we give the necessary background
information on Young functions and Orlicz spaces, and then
prove two extrapolation theorems.  The first one is an unweighted extrapolation used in our applications, Theorem~\ref{RHOrlicz}, and the second is Theorem \ref{thm:wtdOrlicz}, a generalization of Theorem~\ref{thm:wtdOrlicz-special}.  Though some of the proof of Theorem~\ref{RHOrlicz} overlaps with
Theorem~\ref{thm:wtdOrlicz}, we include the proof for two reasons.
First, it is more general due to the range of $p$;
second, the proof makes clear the main ideas while avoiding the
technicalities that arise in the weighted case.  Finally, in
Section~\ref{section:bilinear} we prove our applications in
Theorems~\ref{thm:CZO}, \ref{bilinearp>1}, \ref{thm:bilinear-p<1},
and~\ref{thm:bilinear-CF}. 

Throughout this paper, $n$ will denote the dimension of the space
$\R^n$.  If we write $A\lesssim B$, we mean $A\leq CB$ for some
constant $C$; $A\approx B$ means $A\lesssim B$ And $B\lesssim A$.  Whether implicit or explicit, unless otherwise
specified the constants may depend
on the dimension $n$, $p$, the weights and the operator being studied,
and can change from line to line.  

\begin{remark}
As we were completing this paper, we learned that
Theorem~\ref{RHextrap} was discovered independently by Martell and
Prisuelos~\cite{martell-prisuelos}. 
\end{remark}

\section{$A_p$, $RH_s$ weights and extrapolation} 
\label{section:extrapolation}

\subsection*{Preliminaries about weights}
Hereafter, by a weight we mean a non-negative, locally integrable function.
We begin with a few preliminary facts about $A_p$ and $RH_s$ weights
we will need in this and the following section.  
Beyond the $A_p$ and $RH_s$ classes defined above, we define three
additional weight classes.  We say that a weight $w\in A_1$ if
\[ [w]_{A_1} = \sup_Q \left(\avgint_Q w\,dx\right)\esssup_{x\in Q}\big( w^{-1}(x)\big) <
\infty, \]
where the supremum is taken over all cubes $Q$.
Recall that the $A_p$ classes are nested:  for all $q>p>1$,
$A_1\subset  A_p \subset A_q$.   Analogously, we define
$w\in RH_\infty$ if
\[ [w]_{RH_\infty} =\sup_Q \Big(\esssup_{x\in Q} w(x)\Big)\left(\avgint_Q
  w\,dx\right)^{-1} < \infty, \]
where the supremum is taken over all cubes $Q$; then for all
$r<s<\infty$, $ RH_\infty\subset RH_s \subset RH_r$.  Finally, we let
$A_\infty$ denote the union of the $A_p$ classes:
\[ A_\infty = \bigcup_{1<p<\infty} A_p. \]

There is a close connection between $A_p$ weights and the
Hardy-Littlewood maximal operator.  Given $f\in L^1_{loc}$, define
\[ Mf(x) = \sup_Q \avgint_Q |f|\,dy \cdot \chi_Q(x); \]
where the supremum is taken over all cubes $Q$.

For proofs of the results in the following lemma,
see~\cite{cruz-uribe-neugebauer95,duoandikoetxea01,garcia-cuerva-rubiodefrancia85}.

\begin{lemma} \label{lemma:Ap-wts}
Given a weight w:
\begin{enumerate}

\item $w\in A_\infty$ if and only if there exists $s>1$ such that
$w\in RH_s$;

\item $w\in A_\infty$ if and only if there exist constants
  $0<\alpha,\,\beta<1$ such that given any cube $Q$ and $E\subset Q$
  with $|E|<\alpha|Q|$, $w(E) < \beta w(Q)$;

\item  $w\in RH_s$ for some $1< s<\infty$ if and only if $w^s\in
  A_{\infty}$;

\item if $w\in RH_\infty$, then $w^s\in RH_\infty$ for all $s>0$;

\item given $p>1$,  if $w\in A_1$ then $w^{1-p'}\in RH_{\infty}\cap
  A_p$, and if $w\in RH_{\infty}\cap A_p$ then $w^{1-p'}\in A_1$; 

\item for all $0<r <1$, the function
$(Mw)^r\in A_1$. 

\end{enumerate}
\end{lemma}

\subsection*{Extrapolation}
While a major application of extrapolation is to prove norm inequalities
for operators, it can be applied much more broadly if it is stated
in terms of pairs of functions.  We follow the formulation used
in~\cite{MR2797562}.   Hereafter, $\F= \{(f,g)\}$ will
denote a family of pairs of non-negative, measurable functions that
are not identically $0$. 
Given a fixed family $\F$ and some weighted space $L^p(w)$,  if we
write 
\[ \|f\|_{L^p(w)} \lesssim\|g\|_{L^p(w)}, \qquad (f,g) \in \F, \]
then we mean that this inequality holds for all pairs $(f,g)$ for
which the lefthand term in the inequality is finite.  (This assumption
assures that in the underlying proofs, it is possible to estimate the
norm by duality.)  The constant $C$ can depend on the $A_p$ and/or
$RH_s$ characteristic of $w$, and on $s$ and $p$, but it cannot depend
on the weight $w$ itself.  

In practice, to prove weighted norm inequalities for an operator $T$,
it suffices to consider a family of pairs of functions of the form
$(|Tf|,|f|)$, where $f$ is taken from some suitably chosen dense
family of functions (e.g., $f\in L^\infty_c$).  In order to get the
norm finiteness of the first term, we can replace $|Tf|$ by
$\min(|Tf|,N)\chi_{B(0,N)}$ and then take the limit as $N\rightarrow
\infty$.  

\medskip

We can now prove Theorem~\ref{RHextrap}.  As we noted in the
Introduction, this result is a consequence of the limited range
extrapolation theorem of Auscher and Martell~\cite[Theorem~4.9]{AM}. 

\begin{theorem} \label{chemaextrap} Given $0<s_0<q_0<\infty$ and a
  family $\F =\{(f,g)\}$, suppose there exists 
  $s_0\leq p_0\leq q_0$ such that for all $w\in A_{\frac{p_0}{s_0}}\cap
RH_{(\frac{q_0}{p_0})'}$, 
$$\|f\|_{L^{p_0}(w)}\lesssim\|g\|_{L^{p_0}(w)} \qquad (f,g)\in
\F.$$ 
Then for all $s_0<p<q_0$ and $w\in A_{\frac{p}{s_0}}\cap RH_{(\frac{q_0}{p})'}$,
$$\|f\|_{L^p(w)}\lesssim\|g\|_{L^p(w)} \qquad  (f,g)\in \F.$$ 
\end{theorem}

\begin{proof}[Proof of Theorem~\ref{RHextrap}]  
Fix $0<p<q_0$ and $w\in RH_{(\frac{q_0}{p})'}$; we will show that 
$$\|f\|_{L^p(w)}\lesssim\|g\|_{L^p(w)},\qquad  (f,g)\in \F.$$
By  Lemma~\ref{lemma:Ap-wts}(a),  we can fix $q$ sufficiently large so
that $w\in A_q\cap RH_{(\frac{q_0}{p})'}$.  Fix
$0<s_0<\min(\frac{p}{q},p_0)$.  Then by assumption, for all $v \in 
A_{\frac{p_0}{s_0}}\cap RH_{(\frac{q_0}{p_0})'}\subset  RH_{(\frac{q_0}{p_0})'}$, 
$$\|f\|_{L^{p_0}(v)}\lesssim\|g\|_{L^{p_0}(v)} \qquad (f,g)\in \F. $$
Therefore, the hypotheses of Theorem~\ref{chemaextrap} are satisfied,
and since we have that $w\in A_{\frac{p}{s_0}}\cap
RH_{(\frac{q_0}{p})'}$, the desired inequality holds.
%
\end{proof}

\section{Orlicz Reverse H\"older Extrapolation} 
\label{section:orlicz}

\subsection*{Young functions and Orlicz norms}
Here we gather the basic properties of Young functions and Orlicz
norms that we will use.  We follow~\cite{MR2797562}; for proofs see,
for example,~\cite{rao-ren91}.

A Young function $\Phi :[0,\infty)\to[0,\infty) $ is a convex,
increasing, continuous function such that $\Phi(0) = 0$ and
\[ \lim_{t\rightarrow \infty} \frac{\Phi(t)}{t} = \infty. \]
Young functions are sometimes normalized so that $\Phi(1) = 1$; doing so
simplifies the constants that appear below.  

Given a Young function $\Phi$,  the complementary Young function
$\bar{\Phi}$ is defined by
 \[\bar{\Phi}(s) = \sup_t\{st-\Phi(t)\}\]
If $\Phi(t)=t^p$, $p>1$, then $\bar{\Phi}(t)=t^{p'}$.  More generally, 
\[\Phi^{-1}(t)\bar{\Phi}^{-1}(t) \approx t. \]
Another important example are Orlicz functions of the form
$ \Phi(t) =t^p\log(e+t)^{p-1+\delta}$, $\delta>0$.  In this case we
have that
\[ \bar{\Phi}(t) \approx \frac{t^{p'}}{\log(e+t)^{1+(p'-1)\delta}}. \]

Given a Young function $\Phi$ and a cube $Q$, we define the localized
Orlicz norm
\[ \|f\|_{\Phi, Q} = \inf\bigg\{ \lambda > 0 :
\avgint_Q \Phi\bigg(\frac{|f(x)|}{\lambda}\bigg)\,dx \leq 1
\bigg\}. \]
If $\Phi(t)=t^p$, then we get the localized Lebesgue norm:
\[ \|f\|_{\Phi, Q} = \left(\avgint_Q
  |f|^p\,dx\right)^{\frac{1}{p}}. \]
These norms form an increasing scale:  more precisely, given Young
functions $\Phi$ and $\Psi$, if $\Phi(t) \lesssim \Psi(ct)$, then
\[ \|f\|_{\Phi,Q} \lesssim \|f\|_{\Psi,Q}, \]
with a constant independent of $Q$.

This norm satisfies versions of H\"older's inequality:  given a
Young function $\Phi$ and a cube $Q$,
\[ \avgint_Q |fg|\,dx \leq 2\|f\|_{\Phi,Q}\|g\|_{\bar{\Phi},Q}. \]
More generally, given Young functions $\Phi,\,\Psi,\,\Theta$ such that
$\Phi^{-1}(t)\Psi(t)^{-1} \lesssim \Theta^{-1}(t)$, then
\[ \|fg\|_{\Theta, Q} \leq 2\|f\|_{\Phi,Q}\|g\|_{\Psi,Q}. \]

We now define two growth conditions on Young functions.  
The first bounds a Young function from below.   Given a Young function $\Psi$ and $a>1$, we say that $\Psi$ is an
$a$-Young function if  $\Psi_a(t)=\Psi(t^{\frac{1}{a}})$ is a Young function.
In this case, $t^a \lesssim \Psi(t)$ for large $t$.  
The second condition bounds a Young function from above.  Given $1<p<\infty$, a Young function $\Phi$ satisfies the $B_p$
condition, denoted by $\Phi \in B_p$, if
$$\int_1^\infty\frac{\Psi(t)}{t^p}\,\frac{dt}{t}<\infty.$$
In this case we have $\Psi(t) \lesssim t^p$ for large $t$.  This condition was introduced by 
P\'erez~\cite{perez95} to study the Orlicz maximal operators
$$M_\Psi f(x)=\sup_{Q\ni x}\|f\|_{\Psi,Q}.$$
He proved the following $L^p$ estimate.  

\begin{theorem} \label{thm:orlicz-max}
For all $1<p<\infty$, $\|M_\Psi f\|_{L^p}\lesssim\|f\|_{L^p}$ if and only if $\Psi\in B_p$.
\end{theorem}

 \begin{remark}
In~\cite{perez95} the statement of this result contained the further
hypothesis that $\Psi$ was doubling; however, this was shown to be
superfluous by Liu and Luque~\cite{MR3256181} (see also \cite{Anderson}).  
 \end{remark}

\medskip

Finally, we have that the generalized reverse H\"older class $RH_\Phi$
is contained in the scale of $A_p$ weights.  More precisely, we have the
following lemma that was proved in~\cite{HSV} in a slightly different
form.  For the convenience of the reader we repeat the short proof.  

\begin{lemma} \label{lemma:Ainfty}
Given any Young function $\Phi$, if $w \in RH_\Phi$, then $w\in
A_\infty$.  
\end{lemma}

\begin{proof}
Fix a cube $Q$ and a measurable set $E\subset Q$.  We first estimate
the norm $\|\chi_E\|_{\bar{\Phi},Q}$.   Fix $\lambda>0$ such that
\[ \dashint_Q \bar{\Phi}\left(\frac{\chi_E(x)}{\lambda}\right)\,dx =
1.  \]
Then we have that 
\[ \|\chi_E\|_{\bar{\Phi},Q}= \lambda = \bar{\Phi}^{-1}\left(\frac{|Q|}{|E|}\right)^{-1}. \]
Since  $\Phi^{-1}(t)\bar{\Phi}^{-1}(t)\approx t$,  we get that
\[ \|\chi_E\|_{\bar{\Phi},Q} \simeq \frac{|E|}{|Q|}
\Phi^{-1}\left(\frac{|Q|}{|E|}\right). \]

We now estimate as follows:
\[ w(E) = |Q|\dashint_Q w\chi_E\,dx
\lesssim |Q|\|w\|_{\Phi,Q}\|\|\chi_E\|_{\bar{\Phi},Q}
\lesssim w(Q) \frac{|E|}{|Q|}
\Phi^{-1}\left(\frac{|Q|}{|E|}\right). \]

Since $\Phi$ is a Young function, we have that 
\[ \lim_{s\rightarrow \infty} \frac{s}{\Phi^{-1}(s)} = \lim_{t\rightarrow \infty} \frac{\Phi(t)}{t}=\infty. \]
Therefore, we can find $0<\alpha,\,\beta<1$ such that if
$|E|/|Q|<\alpha$, then $w(E)/w(Q) < \beta$.  Then by
Lemma~\ref{lemma:Ap-wts}(b), $w\in
A_\infty$.
\end{proof}

\subsection*{Extrapolation with generalized reverse H\"older weights}

We can now state and prove our main extrapolation theorems.  Our first
result yields unweighted inequalities.

\begin{theorem} \label{RHOrlicz} 
Given $p_0<q_0$ and $\Psi_0\in B_{(\frac{q_0}{p_0})'}$, suppose that for
all $w\in RH_{\Psi_0}$, 
\begin{equation} \label{eqn:RHOrlicz1}
\|f\|_{L^{p_0}(w)}\leq C\|g\|_{L^{p_0}(w)}, \qquad (f,g)\in \F .
\end{equation}
Then for all $0< p\leq q_0$,
\begin{equation} \label{eqn:RHOrlicz2}
\|f\|_{L^p}\leq C\|g\|_{L^p}, \qquad (f,g) \in \F. 
\end{equation}
If $p_0=q_0$, the same conclusion holds if we assume~\eqref{eqn:RHOrlicz1}
holds whenever $w\in RH_\infty$. 
\end{theorem}

\medskip

Theorem~\ref{RHOrlicz} is actually a consequence of
Theorem~\ref{RHextrap}.  Since $\Psi_0 \in
B_{(\frac{q_0}{p_0})'}$, $\Psi_0(t) \lesssim t^{(\frac{q_0}{p_0})'}$
for $t\geq 1$.  Therefore, if $w\in RH_{(\frac{q_0}{p_0})'}$, 
\[ \|w\|_{\Psi_0,Q} \lesssim \|w\|_{(\frac{q_0}{p_0})', Q} \lesssim
\|w\|_{1,Q}, \]
and so $w\in RH_{\Psi_0}$.    Thus~\eqref{eqn:RHOrlicz1} implies
that~\eqref{eqn:RHExtrapol1} holds, and so
by Theorem~\ref{RHextrap}, for $0<p<q_0$, \eqref{eqn:RHExtrapol2}
holds.  If we take $w=1$ we get \eqref{eqn:RHOrlicz2}.

Despite this, here we give a direct proof.  We do so for two reasons.
First, our proof is ultimately simpler, since it avoids limited
range extrapolation which itself is nontrivial to prove.  Second, as
we noted in the Introduction, our
proof makes clear the main ideas that will be used
in the proof of Theorem~\ref{thm:wtdOrlicz}, which is not a
consequence of Theorem~\ref{chemaextrap}.

\begin{proof} 
We first consider the case $p_0< p\leq q_0$ (the case $p=p_0$ is
obvious by taking $w=1$).  By duality we have that
$$\|f\|_{L^p}=\| |f|^{p_0} \|_{L^{\frac{p}{p_0}}}^{\frac{1}{p_0}}=\Big(\int
|f|^{p_0} h\,dx\Big)^{\frac{1}{p_0}}$$
for some $h\in L^{(\frac{p}{p_0})'}$ with norm one.   Moreover, since the $B_p$ classes are nested we have 
$$\Psi_0\in B_{(\frac{q_0}{p_0})'}\subset B_{(\frac{p}{p_0})'}$$
since $p\leq q_0$.  Hence
$$M_{\Psi_0}:L^{(\frac{p}{p_0})'}\ra L^{(\frac{p}{p_0})'}.$$

We now define a Rubio de Francia iteration algorithm: 
$$\Rf  h=\sum_{k=0}^\infty \frac{M_{\Psi_0}^k h}{2^k\|M_{\Psi_0}\|^k_{L^{(\frac{p}{p_0})'}}},$$
where $M^0_{\Psi_0} h = |h|$.  The operator $\Rf $ satisfies the following properties:
\begin{enumerate}
\item $|h|\leq \Rf h$,
\item $\|\Rf h\|_{L^{(\frac{p}{p_0})'}}\leq 2\|h\|_{L^{(\frac{p}{p_0})'}}$,
\item $M_{\Psi_0}(\Rf h)\leq C\Rf h,$
\item $\Rf h\in RH_{\Psi_0}$.
\end{enumerate}
The first three points are standard:  cf.~\cite{MR2797562}.  To prove
the final point, fix a cube $Q$.  Then we have that
$$\|\Rf h\|_{{\Psi_0},Q}\leq \dashint_Q M_{\Psi_0}(\Rf h)\leq C\dashint_Q \Rf h.$$

We can now estimate as follows:
\begin{align*}
\Big(\int |f|^{p_0} h\,dx\Big)^{\frac{1}{p_0}}
&\leq \Big(\int |f|^{p_0} \Rf h\,dx\Big)^{\frac{1}{p_0}}\\
&\leq C\Big(\int |g|^{p_0} \Rf h\,dx\Big)^{\frac{1}{p_0}} \qquad (\text{since} \ \Rf h\in RH_{\Psi_0})\\
&\leq  C\Big(\int |g|^{p}\Big)^{\frac{1}{p}}\Big(\int \Rf h^{(\frac{p}{p_0})'}\,dx\Big)^{\frac{1}{p_0(\frac{p}{p_0})'}}\\
&\leq 2^{\frac{1}{p_0}}C\Big(\int |g|^{p}\,dx\Big)^{\frac{1}{p}}.
\end{align*}

\bigskip We now consider the case when $0<p<p_0\leq q_0$.  This case
is much simpler and only relies on the maximal operator.  Fix
$r>\frac{1}{p}$ and define
$H= M(g^{\frac{1}{r}})^{\frac{pr}{(p_0/p)'}}$.  Then
$H^{-p_0/p}=M(g^{\frac{1}{r}})^{-a}$ for $a>0$, and so by
Lemma~\ref{lemma:Ap-wts}(e,f),
$H^{-p_0/p}\in RH_\infty\subset RH_{\Psi_0}$.

We can now estimate as follows: by our hypothesis and since the
maximal operator is bounded on $L^{pr}$,
\begin{align*}
\|f\|_{L^p}^p
& = \int_{\R^n} f^p H^{-1}H\,dx \\
& \leq \left( \int_{\R^n} f^{p_0}H^{-p_0/p}\,dx\right)^{p/p_0}
\left(\int_{\R^n} H^{(p_0/p)'}\,dx\right)^{1/(p_0/p)'} \\
& \lesssim \left( \int_{\R^n} g^{p_0}H^{-p_0/p}\,dx\right)^{p/p_0}
\left(\int_{\R^n} M(g^{\frac{1}{r}})^{pr}\,dx \right)^{1/(p_0/p)'} \\
& \leq \left( \int_{\R^n} g^{p_0}
  (g^{\frac{1}{r}})^{r(p-p_0)}\,dx\right)^{p/p_0} 
\left(\int_{\R^n} g^p\,dx \right)^{1/(p_0/p)'} \\
& = \int _{\R^n} g^p\,dx.
\end{align*}

\end{proof}

Now we state and prove Theorem~\ref{thm:wtdOrlicz}, the more general version of Theorem-\ref{thm:wtdOrlicz-special}.
 
\begin{theorem} \label{thm:wtdOrlicz}
Given $0<p_0 <q_0$, suppose that for a fixed $\Psi_0 \in
B_{(\frac{q_0}{p_0})'}$ and all $w \in RH_{\Psi_0}$, 
\begin{equation} \label{eqn:wtdOrlicz1} 
 \|f\|_{L^{p_0}(w)} \lesssim \|g\|_{L^{p_0}(w)}, 
\qquad (f,g) \in \F. 
\end{equation}
If $p_0=q_0$, suppose~\eqref{eqn:wtdOrlicz1} holds for any $w\in RH_\infty$.
If either of the following hold:

\begin{enumerate}

\item $p_0<p<q_0$ and 
$w\in RH_\Psi$, where $\Psi$ is defined by $\Psi_0(t)=\Psi(t^r)$ with $r=\frac{(q_0/p_0)'}{(q_0/p)'} <1$;

\item $p=q_0$ and $w\in RH_\infty$;

\end{enumerate}
then we have that
\begin{equation} \label{eqn:extrapol-concl}
\|f\|_{L^{p}(w)} \lesssim \|g\|_{L^{p}(w)}, 
\qquad (f,g) \in \F. 
\end{equation}
\end{theorem}

\begin{remark}
Notice that as $p$ gets close to $q_0$, then
  $\frac{1}{r} \to \infty$, so the second case is a natural endpoint
  condition.
\end{remark}

\begin{remark}
  In Theorem~\ref{thm:wtdOrlicz} we are not able to prove weighted
  inequalities in the range $0<p<p_0\leq q_0$ analogous to the
  unweighted inequalities in Theorem \ref{RHOrlicz}.   Our proof in
  the unweighted case does not extend to the weighted setting.   This problem seems to be
  much more subtle and will require new techniques.
\end{remark}

\begin{proof}
The proof follows the same outline as the proof of
Theorem~\ref{RHOrlicz}, and we refer to that proof for some details
that are the same.  
We consider each case in turn.  

First suppose that $p_0<p<q_0$; by
duality there exists $h\in L^{(\frac{p}{p_0})'}$, $\|h\|_{L^{(\frac{p}{p_0})'}}=1$, such that
\[
 \|f\|_{L^p(w)}^{p_0} 
  = \left( \int_{\R^n} f^{p_0\frac{p}{p_0}}
  w^{\frac{p_0}{p}\frac{p}{p_0}}\,dx\right)^{\frac{p_0}{p}} \\
 = \int_{\R^n} f^{p_0} w^{\frac{p_0}{p}} h \,dx. 
\]

Since
$\Psi_0(t)=\Psi(t^r)$,  
$\Psi \in B_{(\frac{q_0}{p})'}$:   
by a change of variables we have that
\[ \int_1^\infty
\frac{\Psi(t)}{t^{(\frac{q_0}{p})'}}\frac{dt}{t}
\approx  \int_1^\infty \frac{\Psi_0(t)}{t^{(\frac{q_0}{p_0})'}}\frac{dt}{t} < \infty; \]
the last inequality holds by our assumption that $\Psi_0 \in
B_{(\frac{q_0}{p_0})'}$.  

Now suppose that we have a non-negative function $H$ that satisfies the
following conditions:
\begin{enumerate}

\item $h \leq H$;

\item $\|H\|_{L^{(\frac{p}{p_0})'}} \lesssim \|h\|_{L^{(\frac{p}{p_0})'}} = 1$.

\item $Hw^{\frac{p_0}{p}} \in RH_{\Psi_0}$;

\end{enumerate}
Then by our hypothesis and the properties of $H$ we can estimate as
follows:
\begin{align*}
\int_{\R^n} f^{p_0} w^{\frac{p_0}{p}} h \,dx
& \leq \int_{\R^n} f^{p_0} Hw^{\frac{p_0}{p}}  \,dx \\
& \lesssim \int_{\R^n} g^{p_0} Hw^{\frac{p_0}{p}}  \,dx \\
& \leq \|g^{p_0}w^{\frac{p_0}{p}}\|_{L^{\frac{p}{p_0}}}\|H\|_{L^{(\frac{p}{p_0})'}} \\
& \leq \|g\|_{L^p(w)}^{p_0}.
\end{align*}

\medskip

Therefore, to complete the argument for this case we need to construct
a function $H$ with the desired properties.  We first construct two
auxiliary Young functions.  Let $C(t)=\Psi(t^{\frac{p}{p_0}})$.  We
claim that $w^{\frac{p_0}{p}}\in RH_C$.  Indeed, by Lemma
\ref{lemma:Ainfty}, $w\in A_\infty$ and so by Lemma
\ref{lemma:Ap-wts}(c) we have that
$w^{\frac{p_0}{p}}\in RH_{\frac{p}{p_0}}$.  Therefore, by rescaling
the norm, we have that
\[ \|w^{\frac{p_0}{p}}\|_{C,Q} = \|w\|_{\Psi,Q}^{\frac{p_0}{p}} 
\lesssim \|w\|_{1,Q}^{\frac{p_0}{p}} \lesssim \|w^{\frac{p_0}{p}}\|_{1,Q}. \]

Now define $s>0$ by 
\[ \frac{1}{s} = \frac{1}{r}-\frac{p_0}{p}.  \]
If $1<\frac{1}{s}<\frac{1}{r}$, $\Psi(t^s) = \Psi((t^r)^{s/r}) = \Psi_0(t^{s/r})$, and $s/r>1$ so $B(t)=\Psi(t^s)$ is a Young function; on the other
hand, if $0<\frac{1}{s}\leq 1$ then $s\geq 1$ and $B$ is again a Young
function.  Moreover, in either case we have that $B\in B_{(\frac{p}{p_0})'}$ and hence $M_B$
is bounded on $L^{(\frac{p}{p_0})'}$.    To see this,
first note that 
\[ \frac{1}{s}\left(\frac{p}{p_0}\right)' 
= \left(\frac{q_0-p_0}{q_0-p}- \frac{p_0}{p}\right)\frac{p}{p-p_0} 
=\frac{q_0(p-p_0)}{p(q_0-p)} \frac{p}{p-p_0} = \frac{q_0}{q_0-p} 
= \left(\frac{q_0}{p}\right)'.
\]
Then by a change of variables and the fact that $\Psi\in
B_{(\frac{q_0}{p})'}$,
\[ \int_1^\infty \frac{B(t)}{t^{(\frac{p}{p_0})'}}\frac{dt}{t}
= \int_1^\infty \frac{\Psi(t^s)}{t^{(\frac{p}{p_0})'}}\frac{dt}{t}
\approx \int_1^\infty
\frac{\Psi(t)}{t^{(\frac{q_0}{p})'}}\frac{dt}{t}< \infty.  \]
We can now define $H$ using a Rubio de Francia iteration algorithm:
\begin{equation} \label{eqn:Hdefn}
 H = \Rf h = \sum_{k=0}^\infty \frac{M_B^k h}{2^k
  \|M_B\|_{L^{(\frac{p}{p_0})'}}^k}. 
\end{equation}
Then, arguing as in the proof of Theorem~\ref{RHOrlicz} we have that
$h \leq H$ and $\|H\|_{L^{(\frac{p}{p_0})'}}\leq 2 \|h\|_{L^{(\frac{p}{p_0})'}}$.
This proves properties $(a)$ and $(b)$ above.
Moreover, since $B$ is a Young function, again by the above
argument we have that
\[ M(\Rf h) \leq M_B(\Rf h) \lesssim \Rf h; \]
Thus $H \in A_1 \cap RH_B$.    By the definition of $B$ and $C$ we
have that
\[  C^{-1}(t) B^{-1}(t) = \Psi^{-1}(t)^{\frac{p_0}{p}} \Psi^{-1}(t)^{\frac{1}{s}}
 = \Psi^{-1}(t)^{\frac{1}{r}} = \Psi_0^{-1}(t). \]
Therefore, by the generalized H\"older's inequality and the definition
of $A_1$,
\[ \|Hw^{\frac{p_0}{p}}\|_{\Psi_0,Q} 
\lesssim \|H\|_{B,Q}\|w^{\frac{p_0}{p}}\|_{C,Q}
\lesssim \|H\|_{1,Q}\|w^{\frac{p_0}{p}}\|_{1,Q}
\lesssim \|Hw^{\frac{p_0}{p}}\|_{1,Q}, \]
which proves property $(c)$.  This completes our proof when
$p_0<p<q_0$.

\bigskip

The proof when $p=q_0$ is nearly the same as the previous case; here
we describe the changes.  Fix $w\in RH_\infty$, and let $\Psi_0$ be
any Young function in $B_{(q_0/p_0)'}$.   Let $B=\Psi_0$ and define
$H$ by \eqref{eqn:Hdefn}.  Then $H\in A_1 \cap RH_B$ and satisfies properties $(a)$ and $(b)$ as
before.    To prove $(c)$ note first that by
Lemma~\ref{lemma:Ap-wts}(d), $w^{\frac{p_0}{p}}\in RH_\infty$.  By
this, and then using that $H\in RH_B$ and then that $H\in A_1$,
\[ \|Hw^{\frac{p_0}{p}}\|_{\Psi_0,Q} 
\lesssim \|H\|_{B,Q}\|w^{\frac{p_0}{p}}\|_{1,Q}
\lesssim \|H\|_{1,Q}\|w^{\frac{p_0}{p}}\|_{1,Q}
\lesssim \|Hw^{\frac{p_0}{p}}\|_{1,Q}. \]
Given this function $H$, the remainder of the proof goes through
without change. 
This completes the proof.
\end{proof}

\section{Applications}
\label{section:bilinear}

In this section we give several applications of reverse H\"older
extrapolation to prove weighted norm inequalities.   In spirit, though
not in detail, these applications are similar to those proved via
$A_\infty$ extrapolation in~\cite{cruz-uribe-martell-perez04}.  

\subsection*{Calder\'on-Zygmund operators}
A Calder\'on-Zygmund kernel is a function $K(x,y)$ defined away from the
diagonal $\{(x,y):x=y\}$ that satisfies
$$|K(x,y)|\lesssim |x-y|^{-n} $$
and
\begin{multline} \label{eqn:continuity}
 |K(x,y)-K(x,y+h)| + |K(x,y)-K(x+h,y)|  \\
\leq C\frac{|h|^\epsilon}{|x-y|^{n+\epsilon}}, \quad |x-y|>2|h|. 
\end{multline}
A Calder\'on-Zygmund operator (CZO) is an $L^2$ bounded linear
operator associated to a Calder\'on-Zygmund kernel $K$ such that the
representation
$$Tf(x)=\int_{\R^n}K(x,y)f(y)\,dy$$
holds for all $f\in L^\infty_c$ and $x\not\in \supp(f)$.  

To prove norm inequalities for CZOs we will use the theory of sparse
operators over dyadic grids.  The following is based on the seminal
work of Lerner~\cite{Lern}; the pointwise estimates are due to Conde-Alonso and
Rey~\cite{CondeAlonso:2014vs} and Lacey~\cite{Lacey:2015wf}.  (See
also the recent monograph by Lerner and Nazarov~\cite{LernNaz}, which
uses a slightly different definition of a dyadic grid.)  

By a dyadic grid $\D$ we mean a collection of cubes $\D=\bigcup_k
\D_k$ in $\R^n$ that have the following properties:
\begin{enumerate}  
\item for each $k$,  if $Q\in \D_k$, then $|Q|=2^{-kn}$;

\item  the cubes in $\D_k$ form a partition of $\R^n$;

\item  if $P,\,Q\in \D$, then $P\cap Q = \varnothing$, $P\subset Q$ or
  $Q\subset P$. 
\end{enumerate}
Given a dyadic grid $\D$ we say a subfamily $\Sp\subset \D$ is {\it
  sparse} if for each $Q\in \Sp$ 
\[ \bigg| \bigcup_{\substack{Q'\subset \Sp \\ Q'\subsetneq Q}} Q'
\bigg| \leq \frac{1}{2}|Q|. \]
As a consequence, there exists $E_Q\subset Q$ such that
the family $\{E_Q\}_{Q\in \Sp}$ is pairwise disjoint and there exists
a uniform constant such that $|Q|\leq c|E_Q|.$

  Given a dyadic grid $\D$ and a sparse family $\Sp \subset \D$,
  define a sparse operator by
$$T^\Sp f(x)=\sum_{Q\in \Sp} \left(\,\dashint_Q
  f\,dy\,\right)\chi_Q(x).$$
Sparse operators are positive, linear operators.  Their importance is
that CZOs can be dominated by them pointwise.  

\begin{theorem} \label{thm:dominate}
Given a CZO operator $T$ and a function $f$, there exist $3^n$ dyadic
grids $\{\D^k\}_{k=1}^{3^n}$ and sparse families $\Sp^k\subset \D^k$ such that 
$$|Tf(x)|\lesssim \sum_{k=1}^{3^n} T^{\Sp^k}(|f|)(x)$$
almost everywhere.  The implicit constant depends on the dimension and
the kernel $K$ associated to $T$. 
\end{theorem}

Using sparse operators and reverse H\"older extrapolation, we can
prove our generalization of Theorem~\ref{thm:double-bump}.   

\begin{theorem} \label{thm:CZO}
Let $T$ be a CZO, and fix $ 1<p<\infty$.  Suppose $(u,v)$ is a pair of weights that satisfies
$$\sup_Q \|u\|_{\Phi,Q}\|v^{-1}\|_{\Psi,Q}<\infty,$$
where $\bar{\Phi}\in B_{p'}$ and $\Psi$ is any Young function.  
Then
\begin{equation} \label{eqn:CZO1}
\|(Tf)u\|_{L^p}\lesssim \|M_{\bar{\Psi}}(fv)\|_{L^p}.
\end{equation}
In particular, if $\bar{\Psi}\in B_p$, then 
\begin{equation} \label{eqn:lerner}
\|(Tf)u\|_{L^p}\lesssim \|fv\|_{L^p}.
\end{equation}
\end{theorem}

\begin{remark}
By using results from~\cite{Hytonen:2015gp,Lacey:2015wf},
Theorem~\ref{thm:CZO} can be extended to singular integral operators
that replace~\eqref{eqn:continuity} with a weaker Dini continuity
condition.  Details are left to the interested reader.
\end{remark}

\begin{proof}
When $\bar{\Psi}\in B_p$, \eqref{eqn:lerner} follows immediately from
\eqref{eqn:CZO1}.  To prove this inequality, by
Theorem~\ref{thm:dominate} it will suffice to prove it with $T$
replaced by a sparse operator $T^\Sp$ and with $f$ non-negative.  By
Theorem \ref{RHOrlicz} with $q_0=p$ and $p_0=1$, it will suffice to
show that if  $w\in RH_{\bar{\Phi}}$, then 
$$\|(T^\Sp f) u\|_{L^1(w)}\lesssim \|M_{\bar{\Psi}}(fv)\|_{L^1(w)}.$$

This inequality follows by a straightforward computation using the
properties of a sparse family.  We have that
$$\int_{\R^n}(T^\Sp f )uw\,dx=\sum_{Q\in \Sp}\Big(\,\dashint_Q
f\,dx\Big)\Big(\, \dashint_Q uw\,dx\Big) |Q|.$$ 
Further, since by Lemma~\ref{lemma:Ainfty}, $w\in
RH_{\bar{\Phi}}\subset A_\infty$, 
\[ \|w\|_{\bar{\Phi},Q} |Q| \leq [w]_{RH_{\bar{\Phi}}}w(Q) \leq
Cw(E_Q). \]
Therefore, by the generalized H\"older's inequality, 
\begin{align*}
\sum_{Q\in \Sp}\Big(\,\dashint_Q f\,dx\Big)\Big(\, \dashint_Q uw\,dx\Big) |Q|&\leq \sum_{Q\in \Sp}\|fv\|_{\bar{\Psi},Q} \|v^{-1}\|_{\Psi,Q} \|u\|_{\Phi,Q} \|w\|_{\bar{\Phi},Q} |Q|\\
&\lesssim \sum_{Q\in \Sp}\|fv\|_{\bar{\Psi},Q}  \|w\|_{\bar{\Phi},Q} |Q| \\
&\lesssim \sum_{Q\in \Sp} \|fv\|_{\bar{\Psi},Q} w(E_Q)\\
&\lesssim \int_{\R^n} M_{\bar{\Psi}}(fv) w\,dx.\\
\end{align*}
\end{proof}

\subsection*{Bilinear Calder\'on-Zygmund operators}
The results of the previous section extend naturally to the
multilinear setting. 
A bilinear CZO is defined by the integral formula
$$T(f,g)(x)=\int_{\R^n}K(x,y,z)f(y)g(z)\,dydz \qquad x\notin (\text{supp}\, f)\cap (\text{supp}\, g).$$
for $f,g\in L^\infty_c(\R^n)$ where $K$ is a bilinear Calder\'on-Zygmund kernel:
$$|K(x,y,z)|\lesssim (|x-y|+|x-z|)^{-2n}, \qquad |\nabla K(x,y,z)|\lesssim (|x-y|+|x-z|)^{-2n-1}.$$
Bilinear CZOs can also be dominated pointwise by bilinear sparse
operators.  Again, given a dyadic grid $\D$ and a sparse family
$\Sp\subset \D$, we define
$$T^{\Sp}(f,g)(x)=\sum_{Q\in \Sp} \left(\dashint_Q
  f\,dy\right)\left(\dashint_Q g\,dy\right)\chi_Q(x).$$

The following estimate was proved
in~\cite{CondeAlonso:2014vs,LernNaz}.  

\begin{theorem} \label{thm:bilinear-sparse}
Given a bilinear CZO, $T$ and functions $f$, $g$, there exist $3^n$
dyadic grids $\D_k$ and sparse families $\Sp_k\subset \D_k$ such that 
\[ |T(f,g)(x)| \lesssim \sum_{k=1}^{3^n} T^{\Sp_k}(|f|,|g|)(x). \]
\end{theorem}

Given two Young functions $\Psi_1$ and $\Psi_2$, we define the bisublinear maximal function
$$\M_{\Psi_1,\Psi_2}(f,g)=\sup_{Q\ni x}
\|f\|_{\Psi_1,Q}\|g\|_{\Psi_2,Q}.$$
Clearly we have that $\M_{\Psi_1,\Psi_2}(f,g)(x)\leq M_{\Psi_1}f(x)
M_{\Psi_2}g(x)$, so by H\"older's inequality, if
$\Psi_1\in B_{p_1}$ and $\Psi_2\in B_{p_2}$, then $\M_{\Psi_1,\Psi_2} : L^{p_1}\times
L^{p_2} \rightarrow L^p$.  
We can now state and prove the analog of  Theorem~\ref{thm:CZO} for
bilinear CZOs.  We get two results; in the first we assume $p>1$.  

\begin{theorem} \label{bilinearp>1} 
Let $T$ be a
  bilinear CZO,  fix $1<p_1,p_2<\infty$, and define
  $p=\frac{p_1p_2}{p_1+p_2}$.  Suppose $p>1$ and ,$(u,v_1,v_2)$
  are weights that satisfy
$$\sup_Q
\|u\|_{\Phi,Q}\|v_1^{-1}\|_{\Psi_1,Q}\|v_2^{-1}\|_{\Psi_2,Q}<\infty, $$
where  $\Phi$ is a Young function with $\bar{\Phi}\in B_{p'}$
  and $\Psi_1$, $\Psi_2$ are Young functions.  
Then
$$\|T(f,g)u\|_{L^p}\lesssim \|M_{\bar{\Psi}_1,\bar{\Psi}_2}(fv_1,gv_2)\|_{L^p}.$$
In particular, if $\bar{\Psi}_1\in B_{p_1}$ and $\bar{\Psi}_2\in B_{p_2}$, then 
$$\|T(f,g)u\|_{L^p}\lesssim \|fv_1\|_{L^{p_1}}\|gv_2\|_{L^{p_2}}.$$
\end{theorem}

\begin{proof} 
  As in the proof of Theorem~\ref{thm:CZO}, it will suffice to prove
  the first inequality; the second is an immediate corollary.  And
  again, it will suffice to prove this for a bilinear sparse operator
  $T^\Sp$ and non-negative $f,\,g$.  
By  Theorem \ref{RHOrlicz} with $q_0=p$ and $p_0=1$ we only need to
prove a weighted $L^1$ inequality.

Fix
$w\in RH_{\bar{\Phi}}$;  then we can essentially repeat the previous argument:
\begin{align*}
\int_{\R^n}T^\Sp(f,g)wu\,dx&=\sum_{Q\in \Sp}\Big(\,\dashint_Q f\,dx\Big)\left(\dashint_Q g\,dx\right)\Big(\, \dashint_Q uw\,dx\Big) |Q|\\
&\leq \sum_{Q\in \Sp}\|fv_1\|_{\bar{\Psi}_1,Q} \|gv_2\|_{\bar{\Psi}_2,Q}\|v_1^{-1}\|_{\Psi_1,Q}\|v_2^{-1}\|_{\Psi_2,Q}  \|u\|_{\Phi,Q} \|w\|_{\bar{\Phi},Q} |Q|\\
&\lesssim \sum_{Q\in \Sp}\|fv_1\|_{\bar{\Psi}_1,Q} \|gv_2\|_{\bar{\Psi}_2,Q}  \|w\|_{\bar{\Phi},Q} |Q| \\
&\lesssim\sum_{Q\in \Sp} \|fv\|_{\bar{\Psi}_1,Q}  \|gv_2\|_{\bar{\Psi}_2,Q}w(E_Q)\\
&\lesssim\int_{\R^n} M_{\bar{\Psi}_1,\bar{\Psi}_2}(fv_1,gv_2) w\,dx.\\
\end{align*}
\end{proof}

Surprisingly, when $p\leq 1$ we do not need an Orlicz bump on the
weight $u$:  it suffices to take the localized $L^p$ norm.  

\begin{theorem} \label{thm:bilinear-p<1}
Let $T$ be a
  bilinear CZO,  fix $1<p_1,p_2<\infty$, and define
  $p=\frac{p_1p_2}{p_1+p_2}$.  Suppose $p\leq 1$ and $(u,v_1,v_2)$ are weights that satisfy
$$\sup_Q \left(\dashint_Q
  u^p\,dx\right)^{\frac{1}{p}}\|v_1^{-1}\|_{\Psi_1,Q}\|v_2^{-1}\|_{\Psi_1,Q}<\infty, $$
where $\Psi_1$, $\Psi_2$  are Young functions.  
Then
$$\|T(f,g)u\|_{L^p}\lesssim \|M_{\bar{\Psi}_1,\bar{\Psi}_2}(fv_1,gv_2)\|_{L^p}.$$
In particular, if $\bar{\Psi}_1\in B_{p_1}$ and $\bar{\Psi}_2\in B_{p_2}$ then 
$$\|T(f,g)u\|_{L^p}\lesssim \|fv_1\|_{L^{p_1}}\|gv_2\|_{L^{p_2}}.$$
\end{theorem}

\begin{proof} 
The proof is more straightforward than the proof of Theorem~\ref{bilinearp>1}
since we do not need to use extrapolation.  Again, we will prove it
for a sparse bilinear operator $T^\Sp$ and a pair of non-negative
functions $f,\,g$.  Since $0<p\leq 1$, by convexity we have the
pointwise inequality 
$$T^S(f,g)^p\leq \sum_{Q\in \Sp}\left[\left(\dashint_Q
    f\,dx\right)\left(\dashint_Q g\,dx\right)\right]^p\chi_Q.$$
Therefore, proceeding as we did above,
\begin{align*}\int_{\R^n}(T^\Sp(f,g)u)^p\,dx&\leq\sum_{Q\in \Sp}\left[\left(\dashint_Q
    f\,dx\right)\left(\dashint_Q g\,dx\right)\right]^p\left(\dashint_Q u^p\,dx\right)|Q|\\
   &\leq   \sum_{Q\in \Sp}(\|fv_1\|_{\bar{\Psi}_1,Q} \|gv_2\|_{\bar{\Psi}_2,Q}\|v_1^{-1}\|_{\Psi_1,Q}\|v_2^{-1}\|_{\Psi_2,Q})^p\left(\dashint_Q u^p\,dx\right)|Q| \\
   &\lesssim \sum_{Q\in \Sp}(\|fv_1\|_{\bar{\Psi}_1,Q} \|gv_2\|_{\bar{\Psi}_2,Q})^p|Q| \\
   &\lesssim\sum_{Q\in \Sp}(\|fv_1\|_{\bar{\Psi}_1,Q} \|gv_2\|_{\bar{\Psi}_2,Q})^p|E_Q| \\
   &\lesssim\int_{\R^n} M_{\bar{\Psi}_1,\bar{\Psi}_2}(fv_1,gv_2)^p\,dx.
  \end{align*}
\end{proof}

\subsection*{Bilinear fractional integral operators}

Recall from the Introduction that, given $0<\alpha<n$, we define the bilinear fractional integral operator
$$BI_\al(f,g)(x)=\int_{\R^n} \frac{f(x-y)g(x+y)}{|y|^{n-\al}}\,dy$$
and bilinear fractional maximal operator
$$BM_\al(f,g)(x)=\sup_{r>0}\frac{1}{(2r)^{n-\al}}\int_{[-r,r]^n}
|f(x-y)g(x+y)|\,dy.$$ 

Also recall the following, less singular version of the bilinear fractional
integral operator,
$$\mathcal
I_\al(f,g)(x)=\int_{\R^{2n}}\frac{f(y)g(z)}{(|x-y|+|x-z|)^{2n-\al}}\,dydz,
\qquad 0<\al<2n,$$ 
and the associated maximal operator
$$\mathcal M_\al(f,g)(x)=\sup_{Q\ni x}|Q|^{\frac{\al}{n}}\dashint_Q
|f(y)|\,dy \cdot \dashint_Q |g(z)|dz.  $$

 A similar calculation to that in the linear case shows that $BM_\al(f,g)\lesssim BI_\al(f,g)$ and
$\mathcal M_\al(f,g)\lesssim\mathcal I_\al(f,g)$ when $f,g\geq 0$; moreover, it was
shown in \cite{M1} (via $A_\infty$ extrapolation) that for
$0<p<\infty$ and $w\in A_\infty$,
$$\|\mathcal I_\al(f,g)\|_{L^p(w)}\leq C\|\mathcal M_\al(f,g)\|_{L^p(w)}.$$

Here we use extrapolation to give a new proof of the following analogous inequality
for $BI_\alpha$ and $M_\alpha$.  This result was first proved in~\cite{M2}.

\begin{theorem} \label{thm:bilinear-CF}
Given $0<p\leq 1$ and $w\in RH_{(\frac{1}{p})'}$, then 
$$\|BI_\al(f,g)\|_{L^p(w)}\leq C\|\mathcal M_\al(f,g)\|_{L^p(w)}.$$
\end{theorem}

\begin{proof} 
  Our proof is similar in parts to the proof
  of~\cite[Theorem~1.8]{M2}, so we will only sketch the details. We
  will prove that the hypotheses of Theorem~\ref{RHextrap} are
  satisfied when $p_0=q_0=1$:  i.e., we will
  show that if $w\in RH_\infty$, then  we have
$$\int_{\R^n}BI_\al(f,g) w\,dx\lesssim \int_{\R^n} \M_\al(f,g) w\,dx$$
for non-negative functions $f$ and $g$.  

In \cite[Theorem 3.2]{M2} it was shown that $BI_\al$ is dominated pointwise
by the  dyadic operator
$$BI_\al(f,g)(x)\lesssim BI_\al^\D(f,g)(x):= \sum_{Q\in \D}
\frac{|Q|^{\frac{\al}{n}} }{|Q|}\int_{|y|\leq
  \ell(Q)}{f(x-y)g(x+y)}\,dy\cdot \chi_Q(x),$$
where $\D$ is the standard dyadic grid.
Now let $w\in RH_\infty$; then we estimate as follows:
\begin{multline*}
\int_{\R^n}BI_\al^\D(f,g) w\,dx=\sum_{Q\in \D}\frac{|Q|^{\frac{\al}{n}} }{|Q|}\int_Q\int_{|y|\leq \ell(Q)}{f(x-y)g(x+y)w(x)}\,dydx \\
\lesssim \sum_{Q\in \D}\frac{|Q|^{\frac{\al}{n}}
}{|Q|}(\,\sup_{Q}w)\int_Q\int_{|y|\leq \ell(Q)}{f(x-y)g(x+y)}\,dydx. 
\end{multline*}
If we make the change of variables $u=x-y$, $v=x+y$ and use the
$RH_\infty$ condition on $w$, then 
\begin{equation}\label{BIbdd}\int_{\R^n}BI_\al^\D(f,g)w\,dx\lesssim \sum_{Q\in \D}{|Q|^{\frac{\al}{n}} }\,\left(\dashint_{3Q} f\,dx\right)\left(\dashint_{3Q} g\,dx\right)\left(\int_Qw\,dx\right).\end{equation}

This sum is similar to the sum that appeared in the proof of Theorem
\ref{bilinearp>1} except that it is over all dyadic cubes.
However, we will bound it by a sum over a sparse family.
Fix $a>1$; the exact value will be chosen later.  Fix $k\in \Z$ and
let
$$\Ca^k=\left\{Q\in \D: a^k<\left(\dashint_{3Q}
    f\,dx\right)\left(\dashint_{3Q} g\,dx\right)\leq
  a^{k+1}\right\}.$$ 
Let $\Sp^k$ be all cubes in $\D$ that are maximal with respect to inclusion and satisfy
$$a^k<\left(\dashint_{3Q} f\,dx\right)\left(\dashint_{3Q} g\,dx\right).$$
(By an approximation argument we may assume $f$ and $g$ are bounded
and have compact support, so such maximal cubes exist.)  It is clear
that every $Q\in \Ca^k$ is a subset of a unique cube in $\Sp^k$.  We
can now estimate the righthand side of inequality \eqref{BIbdd} as
follows:
\begin{align*}
& \sum_{Q\in \D}{|Q|^{\frac{\al}{n}} }\,\left(\dashint_{3Q}
  f\,dx\right)\left(\dashint_{3Q}
  g\,dx\right)\left(\int_Qw\,dx\right) \\ 
&  \qquad \qquad=\sum_{k\in \Z}\sum_{Q\in \Ca^k}{|Q|^{\frac{\al}{n}} }\,\left(\dashint_{3Q} f\,dx\right)\left(\dashint_{3Q} g\,dx\right)\left(\int_Qw\,dx\right)\\
&  \qquad \qquad\leq \sum_{k\in \Z}a^{k+1}\sum_{Q\in \Ca^k}{|Q|^{\frac{\al}{n}} }\int_Qw\,dx\\
&  \qquad \qquad= \sum_{k\in \Z}a^{k+1}\sum_{P\in \Sp^k }\sum_{{Q\in \D(P)}}{|Q|^{\frac{\al}{n}} }\int_Qw\,dx\\
&  \qquad \qquad=\sum_{k\in \Z}a^{k+1}\sum_{P\in \Sp^k }\sum_{j=1}^\infty\sum_{\substack{Q\in \D(P)\\ \ell(Q)=2^{-j}\ell(P)}}{|Q|^{\frac{\al}{n}} }\int_Qw\,dx\\
&  \qquad \qquad\lesssim \sum_{k\in \Z}a^{k+1}|Q|^{\frac{\al}{n}}\sum_{P\in \Sp^k }\int_Pw\,dx\\
&  \qquad \qquad\lesssim \sum_{Q\in \Sp}{|Q|^{\frac{\al}{n}}
  }\,\left(\dashint_{3Q} f\,dx\right)\left(\dashint_{3Q}
  g\,dx\right)\left(\int_Qw\,dx\right), 
\end{align*}
where in the last line  we let $\Sp=\bigcup_k\Sp^k.$  

We claim that $\Sp$ is a sparse set.  To see this, let 
$$\Omega_k=\bigcup_{Q\in \Sp^k}Q; $$
then $\Omega_{k}\supseteq \Omega_{k+1}$ and given $Q\in \Sp^k$ we have 
\begin{multline*}\Big|\bigcup_{\substack{Q'\in\Sp\\Q'\subsetneq Q}}Q'\Big|=|Q\cap \Omega_{k+1}|\leq |\{x:\M(f\chi_{3Q},g\chi_{3Q})(x)>a^{k+1}\}|\\
\leq \left[\frac{C}{a^{k+1}}\Big(\int_{3Q} f\,dx\Big)\Big(\int_{3Q} g\,dx\Big)\right]^{\frac{1}{2}}\lesssim \frac{C}{a^{\frac{1}{2}}}|Q|.
\end{multline*}
The second inequality follows from the fact that $\M : L^1(\R^n)\times
L^1(\R^n)\ra L^{1/2,\infty}(\R^n)$.  But then, if  we choose $a$
sufficiently large,  we  get that $\Sp$ is sparse.  

We can now complete the proof.  Since $w\in RH_\infty$ we have
$w(Q)\lesssim w(E_Q)$ ; since the sets $\{E_Q\}$ are disjoint,
\begin{multline*}
\sum_{Q\in \Sp}{|Q|^{\frac{\al}{n}} }\,\left(\dashint_{3Q} f\,dx\right)\left(\dashint_{3Q} g\,dx\right)w(Q)\lesssim \sum_{Q\in \Sp}{|Q|^{\frac{\al}{n}} }\,\left(\dashint_{3Q} f\,dx\right)\left(\dashint_{3Q} g\,dx\right)w(E_Q)\\
\leq \sum_{Q\in \Sp}\int_{E_Q}\M_\al(f,g)w\,dx\leq \int_{\R^n}\M_\al(f,g) w\,dx.
\end{multline*}
If we combine the above estimates, we get the desired inequality and
the proof is complete.
\end{proof}

\bibliographystyle{plain}
\bibliography{RHExtrapolation}

\end{document}